\documentclass{article}

\usepackage[english]{babel}

\usepackage[letterpaper,top=2cm,bottom=2cm,left=3cm,right=3cm,marginparwidth=1.75cm]{geometry}

\usepackage{amsmath,amsfonts,amsthm}
\usepackage{amssymb,graphicx,epsfig,mathtools}
\usepackage[colorlinks=true, allcolors=blue]{hyperref}

\newtheorem{theorem}{Theorem}[section]
\theoremstyle{definition}
\newtheorem{lemma}[theorem]{Lemma}
\newtheorem{claim}[theorem]{Claim}

\newtheorem{question}[theorem]{Question}
\newtheorem{prop}[theorem]{Proposition}

\theoremstyle{remark}

\newcommand{\dbar}{\kern-.1em{\raise.8ex\hbox{ -}}\kern-.6em{d}}

\def \be{\begin{equation}}
\def \ee{\end{equation}}

\newcommand{\sphere}{\mathbb{S}^{n-1}}
\newcommand{\R}{\mathbb{R}}

\title{Monotonicity of the Gaussian measure under Banaszczyk transforms}
\author{Maud Szusterman}

\begin{document}
\maketitle

\begin{abstract}

In the proof of his famous $5K$-theorem, W. Banaszczyk introduced a transformation of convex bodies for which the Gaussian measure is monotone. In this note, we present a simplified proof of this monotonicity by slightly modifying Banaszczyk’s transform, so that it interacts smoothly with Ehrhard symmetrizations, thereby yielding a somewhat easier proof of the $5K$-theorem.
\end{abstract}

\section{Introduction}
\label{section:intro}
\emph{Notations :} For $p \in [1,+\infty]$, denote $||\cdot||_p$ the $\ell_p$-norm on $\R^n$, whose associated unit ball we denote $B_p^n=\{(x_1,...,x_n)\in \R^n : |x_1|^p +... + |x_n|^p\leq 1\}$. Denote $\gamma_n$ the (standard, centered) Gaussian measure on $\R^n$, i.e. the probability measure on $\R^n$ with density $\frac{d\gamma_n(x)}{dx}=(2\pi)^{-n/2} e^{-|x|^2/2}$, where $|x|:=||x||_2$. Denote $\Phi(t):=\gamma_1((-\infty,t])=\int_{-\infty}^t e^{-u^2/2} \frac{du}{\sqrt{2\pi}}$, the cumulative distribution of $\gamma_1$, and denote $\Phi^{-1} : [0,1]\to \R \cup \{\pm \infty\}$ its inverse. Let $\mathcal{K}_n$ denote the set of closed convex subsets of $\R^n$.

Given two convex bodies $U,V\subset \R^n$, the latter containing the origin in its interior, a natural question, already studied by Spencer \cite{Spencer63}, is to determine the least $\beta>0$ such that for any finite sequence of vectors $u_1, ... , u_t$ picked in $U$, there exists signs $\epsilon_i=\pm 1$ such that the signed sum $\epsilon_1 u_1 +... + \epsilon_t u_t$ lies in $\beta V$. More formally, given a convex body $V\subset \R^n$, such that $0\in \text{int}(V)$, the vector balancing constant $vb(U,V)$ of a compact set $U\subset \R^n$ with respect to $V$ is defined as
$$vb(U,V)=\sup_{t\geq 1} \max_{u_1, ... , u_t\in U} \min \left\{\left|\left|\epsilon_1 u_1+\dots +\epsilon_t u_t \right|\right|_V \hspace{1mm} | \epsilon\in \{\pm 1\}^t\right\}$$
where $\lvert\lvert\cdot\rvert\rvert_V$ denotes the Minkowski gauge associated to $V$.
A closely related quantity is $\beta(U,V)=\max_{u_1, ... , u_n\in U} \min \left\{\lvert\lvert\sum_{i=1}^n \epsilon_i u_i \rvert\rvert_V : \hspace{1mm} \varepsilon\in \{\pm 1\}^n\right\}$, where $n=\text{dim}(V)$. In fact, $vb$ and $\beta$ are comparable: one has, independently of the dimension: $\beta(U,V)\leq vb(U,V)\leq 2 \beta(U,V)$ (see \cite{LSV87}). Known examples include $\beta(B^n_2, B^n_2)=\sqrt{n}$, $\beta(B_1^n,B^n_{\infty})\leq 2$ (due to Beck-Fiala \cite{BF81}), and $\beta(B_{\infty}^n, B_{\infty}^n)=O(\sqrt{n})$ (due to Gluskin \cite{Gl89}, and independently, to Spencer \cite{Sp85}). The Koml\'os conjecture asks whether $\beta(B_2^n, B_{\infty}^n)$ remains bounded as the dimension $n$ grows :
$\beta(B_2^n,B_{\infty}^n)=O(1)$? The current best upper bound, due to Banaszczyk \cite{Ban98}, states that $\beta(B_2^n, B_{\infty}^n)=O(\sqrt{\log n})$. Another interesting question, given a known upper bound $\beta(U,V)\leq \beta^+$, is whether one can achieve it efficiently, i.e. does there exist some algorithm which, given $u_1, ...,u_t$ as input, efficiently produces some $\varepsilon\in \{\pm 1\}^t$ such that $||\epsilon_1 u_1 +... +\epsilon_t u_t||_V \leq \tilde{O}(\beta^+)$? In \cite{Bansal}, Bansal et al. answered this question in the affirmative, for $(U,V)=(B_2^n, B_{\infty}^n)$, and $\beta^+=\sqrt{\log n}$, by giving a randomized algorithm (based on the so-called \emph{Gram-Schmidt random walk}) which produces in polynomial time a sequence of signs $\epsilon_i$, which matches Banaszczyk's upper bound (with high probability). In this note, we do not address the question of efficiency. Here our aim is to revisit the original proof by Banaszczyk (1998), of his $5K$-theorem (stated as Theorem \ref{Ban5K} below), and offer some simplification in the proof.
\begin{theorem}[Theorem 1, \cite{Ban98}]
\label{Ban5K}
Let $K\subset \R^n$ be a closed convex set, $\gamma_n(K) \geq \frac{1}{2}$. Let $u_1, ... , u_t\in B_2^n$ (with $t\geq 1$ arbitrary). Then there exists $\epsilon\in \{\pm 1\}^t$ such that $\epsilon.u :=\epsilon_1 u_1+...+\epsilon_t u_t \in 5K$.
\end{theorem}

For $n\geq 2$, one has $\gamma_n(tB_{\infty}^n)\geq \frac{1}{2}$ for $t\geq \sqrt{2\log(n)}$, and so Theorem \ref{Ban5K} implies,  (with constant $C=5\sqrt{2}$), the  upper bound $vb(B_2, B_{\infty}) \leq C \sqrt{\log n}$.

In order to prove Theorem \ref{Ban5K}, the following transform of convex sets is introduced (\cite{Ban98}): given $K\in \mathcal{K}_n$, and $u\in \R^n$, a non-zero vector, define $K*u:=(K+[-u,u]) \cap (Z_B +\R u)$, where $Z_B=Z_B(K,u)\subset u^{\perp}$ is the convex set of those $y\in u^{\perp}$ for which $(K-y)\cap \R u$ has (euclidean) length at least $2||u||_2$. If $I=[a,b]$ is an interval of length $b-a\geq 2r$, then $(I-r)\cup (I+r)=[a-r,b+r]$ is convex. Therefore $K*u$ can equivalently be defined as $K*u=\left((K-u) \cup (K+u)\right) \cap \left(Z_B +\R u\right) $. In particular, $K*u \subset (K-u)\cup (K+u)$.

The following monotonicity statement (Theorem 3 in \cite{Ban98}), is key for proving Theorem \ref{Ban5K} using Banaszczyk's transform $K*u$.
\begin{theorem}[Theorem 3, \cite{Ban98}]
\label{thm3Ban98}
Let $K\subset \R^n$ be a closed convex set, such that $\gamma_n(K)\geq \frac{1}{2}$. Let $u\in \R^n$, with $||u||_2\leq \frac{1}{5}$. Then $\gamma_n(K*u)\geq \gamma_n(K)\geq \frac{1}{2}$.
\end{theorem}

Our contribution is to provide more simple arguments to derive Theorem \ref{thm3Ban98}, by studying a Gaussian variant of $K*u$, which we shall denote $K\circ u$, and which is such that $K\circ u\subset K*u$ for any $K, u$ (see (\ref{dfn:gauss}) on page 3 for a precise definition of $K\circ u$). We prove :
\begin{theorem}
\label{gaussianmono}
Let $K\subset \R^n$ be a closed convex set, such that $\gamma_n(K)\geq \frac{1}{2}$. Let $u\in \R^n$, with $||u||_2\leq \frac{1}{7}$. Then $\gamma_n(K\circ u)\geq \gamma_n(K)\geq \frac{1}{2}$.
\end{theorem}
which implies Theorem \ref{thm3Ban98}, and hence Theorem \ref{Ban5K}, up to a factor $7/5$. A particular case of Theorem \ref{gaussianmono} is when $K$ is an affine half-space; this case is equivalent to the following statement.

\begin{lemma}
    \label{halfplanelem} Let $H\subset \R^2$ be a halfplane containing $0_{\R^2}$. Let $u\in \R^2$, with $||u||_2\leq \frac{1}{7}$. Then $\gamma_2(H\circ u)\geq \gamma_2(H)\geq \frac{1}{2}$.
\end{lemma}

To further motivate our study of $K\circ u$, let us briefly quote \cite{Ban98}, so as to recall here why Theorem \ref{thm3Ban98} implies Theorem \ref{Ban5K}.

Hence let $K_0\subset \R^n$ be a convex body of Gaussian measure $\gamma_n(K_0)\geq \frac{1}{2}$ , and let $u_1, ... , u_t\in \R^n$, $||u_i||_2\leq \frac{1}{5}$. Define the sequence of convex bodies: $K_1=K_0*u_1, K_2=K_1*u_2, \dots , K_t=K_{t-1}*u_t$.
By iteratively applying Theorem \ref{thm3Ban98} to the $K_{i-1}$, $i\geq 1$ (and to the $u_i$), one finds that $\frac{1}{2}\leq \gamma_n(K_0) \leq \gamma_n(K_0*u_1)=\gamma_n(K_1)\leq \gamma_n(K_2) \leq ... \leq \gamma_n(K_t)$.
In particular, since $K_t$ is closed and convex, this implies that $0\in K_t$.

On the other hand, by definition one always has $K*w\subset (K-w) \cup(K+w)$, and so 
$$K_2=K_1*u_2 \subset (K_1-u_2)\cup (K_1+u_2) \subset \cup_{\epsilon \in \{\pm 1\}^2} (K_0 -\epsilon_1 u_1 - \epsilon_2 u_2) \hspace{2mm}, \dots $$
and thus $0\in K_t \subset \cup_{\epsilon} (K_0-\epsilon.u)$ where $\epsilon$ varies through $\{\pm 1\}^t$, and where $\epsilon.u:=\epsilon_1 u_1+...+\epsilon_t u_t$.

\vspace{1mm}
Therefore Theorem \ref{gaussianmono} implies a $7K$-version of Theorem \ref{Ban5K}.
The rest of this note is dedicated to proving Theorem \ref{gaussianmono}, and is organised as follows. Section \ref{section:reductions} gives a formal definition of $K\circ u$, and then shows that to derive Theorem \ref{gaussianmono} for all $K$, it suffices to derive the particular case of half-spaces, or, equivalently (since $\gamma_n=\gamma_2 \otimes \gamma_{n-2}$), the case of half-planes (see Lemma \ref{halfplanelem} below). Section \ref{section:halfplane} is dedicated to proving Lemma \ref{halfplanelem}, i.e. the case of half-planes.

While Section \ref{section:halfplane} is short and elementary, the heart of the proof lies in Section \ref{section:reductions}, i.e. in reducing the monotonicity question to Lemma \ref{halfplanelem}. The reductions proceed in 3 steps : first reduce to hypographs, then to \emph{decreasing} planar hypographs, then to half-planes. For these reductions, our main tools are the Gaussian isoperimetric inequality, and Ehrhard symmetrizations.


The Gaussian isoperimetric inequality, due to Sudakov and Tsirelson \cite{SudakovTsirelson} and, independently, to Borell \cite{borell75}, states that amongst all convex sets of fixed (Gaussian) volume $v>0$, halfspaces (and only halfspaces) minimize the \emph{Gaussian surface area}, in the following sense:
\begin{prop}
\label{prop:gaussisop}(see \cite{borell75}, \cite{SudakovTsirelson})
Let $d\geq 1$ and let $K\subset \R^d$ be a closed convex set, $\gamma_d(K)>0$. 
Then, for any $\epsilon>0$, $\gamma_d(K+\epsilon B_2^d) \geq \Phi(\Phi^{-1}(\gamma_d(K))+\epsilon)$, with equality if and only if $K$ is a halfspace.
\end{prop}
We now state Ehrhard's inequality (from which Proposition \ref{prop:gaussisop} can be derived,
see for instance Latała \cite{latala}, page 816).
\begin{prop}[Theoreme 3.2, p. 292 \cite{ehrhard}]
\label{prop:ehr}
Let $d\geq 1$, and let $A, B\subset \R^d$ be closed convex sets with non-empty interior. Then $(\lambda \mapsto \Phi^{-1}(\gamma_d((1-\lambda) A+\lambda B)))$ is concave on $[0,1]$.
\end{prop}
Given $K\subset \R^n$ a closed convex set, and $y\in e_n^{\perp}  \hspace{1mm}( \approx \R^{n-1})$, let $I_y=(K-y)\cap \R e_n$. By convexity of $K$, if $y_0, y_1\in e_n^{\perp}$ and $\lambda\in [0,1]$, then $I_{\lambda}:=I_{(1-\lambda) y_0+\lambda y_1} \supseteq (1-\lambda) I_{y_0}+\lambda I_{y_1}$. Proposition \ref{prop:ehr}, $d=1$, gives $\Phi^{-1}(\gamma_1(I_{\lambda})) \geq (1-\lambda) \Phi^{-1}(\gamma_1(I_0))+\lambda \Phi^{-1}(\gamma_1(I_1))$. This implies that the following set, is convex
\begin{equation}
\label{eqn:ehr}
\mathcal{E}(K):=\{ (y,z)\in \R^{n-1} \times \R : z\leq \Phi^{-1}(\gamma_1(I_y)) \}.
\end{equation}
The set $\mathcal{E}(K)$ is usually called Ehrhard's symmetrization of $K$ in direction $e_n$. Now, for $z\in \R$, denote $K_z=(K-ze_n) \cap e_n^{\perp}$ the $(n-1)$-dimensional section of $K$, at height $z$ (or rather, its projection onto $e_n^{\perp}$). Another Ehrhard's symmetrization of $K$, is the following planar set :
\begin{equation}
\label{eqn:ehr2}
\mathcal{E}_2(K)=\{(x,z)\in \R^2 : x\leq \Phi^{-1}(\gamma_{n-1}(K_z))\}.
\end{equation}
By convexity of $K$, if $z_0, z_1\in \R$ and $z_{\lambda}:=(1-\lambda)z_0+\lambda z_1$, then $K_{z_{\lambda}} \supseteq (1-\lambda) K_{z_0} +\lambda K_{z_1}$. Hence, by Proposition \ref{prop:ehr}, with $d=n-1$, we see that $\mathcal{E}_2(K)$ is convex as well.

By the product structure of $\gamma_n$ (and of $\gamma_2$), it is immediate from the definitions (\ref{eqn:ehr}) and (\ref{eqn:ehr2}), that $\mathcal{E}$ and $\mathcal{E}_2$ preserve the Gaussian measure. In the next paragraph, we shall denote $\Lambda:=\mathcal{E}(K)$, and $L=\mathcal{E}_2(\mathcal{E}(K))=\mathcal{E}_2(\Lambda)$. In particular, $\gamma_2(L)=\gamma_2(\mathcal{E}_2(\Lambda))=\gamma_n(\Lambda)$, and $\gamma_n(\Lambda)=\gamma_n(\mathcal{E}(K))=\gamma_n(K)$.

To summarize (beforehand) section \ref{section:reductions}: one can assume that $u=r e_n$, and we should show that
\begin{equation}
\label{eqn:summ}
\gamma_n(K\circ (re_n)) \geq \gamma_n(\Lambda \circ (re_n)) \geq \gamma_2(L\circ (re_2)) \geq \gamma_2(H\circ (re_2))
\end{equation}
for some halfplane $H\subset \R^2$ of measure $\gamma_2(H)\geq \gamma_2(L)=\gamma_n(K)\geq \frac{1}{2}$, so that
$\gamma_n(K\circ (re_n))-\gamma_n(K)\geq \gamma_2(H \circ (re_2))-\gamma_2(H)$, and hence Theorem \ref{gaussianmono} is implied by its particular case Lemma \ref{halfplanelem}. The (one-dimensional) Gaussian isoperimetric inequality will justify the first inequality in (\ref{eqn:summ}), our choice of transform $K\circ u$ will easily yield the second inequality, while the planar Gaussian isoperimetric inequality (Proposition \ref{prop:gaussisop}, $d=2$), and the assumption $\gamma_n(K)\geq \frac{1}{2}$, will give the third one.

In the next section, we give the definition of $K\circ u$ (a Gaussian variant of Banaszczyk's transform $K*u$), before explaining the inequalities (\ref{eqn:summ}) in more details (see Steps (\ref{eqn:step1}, \ref{eqn:step2}, \ref{eqn:step3}) below). Then, in section \ref{section:halfplane}, we prove Lemma \ref{halfplanelem}, closing the proof of Theorem \ref{gaussianmono}.

\section{Reducing to the case of half-planes}
\label{section:reductions}

We define the transform $K\circ u$ as follows: if $u=0$, then set $K\circ 0=K*0=K$. Else, let $r=||u||_2>0$, and set $Z=Z_K(u)=\{y\in u^{\perp} : \gamma_1(I_y) \geq p_r\}$ where $I_y=(K-y)\cap \R u$ and where $p_r:=\gamma_1([-r,r])$. In particular, since $\frac{d\gamma_1}{dx}$ is even and radially decreasing, note that $Z\subset Z_B=\{y\in u^{\perp} : \lambda_1(I_y) \geq 2 r\}$, where $\lambda_1(I)$ denotes the Euclidean length of interval $I$.
We set
\begin{equation}
\label{dfn:gauss}
K\circ u:=(K+[-u,u])\cap (Z_K(u)+\R u) \subset (K+[-u,u]) \cap (Z_B +\R u)=K*u
\end{equation}
(note Definition (\ref{dfn:gauss}) is consistent with $K\circ 0=K$) where $K*u$ denotes Banaszczyk's transform (discussed in section \ref{section:intro}). In particular, $K\circ u\subset (K-u) \cup (K+u)$, for any $K$ (convex) and any $u\in \R^n$, which is why Theorem \ref{gaussianmono} implies a $7K$-version of Theorem \ref{Ban5K}, as recalled in the introduction.

Now we shall argue why proving Lemma \ref{halfplanelem} is sufficient for proving Theorem \ref{gaussianmono}.
Let us give an overview before going into details: the classical Ehrhard symmetrization, $\mathcal{E}(K)$, together with the one-dimensional Gaussian isoperimetric inequality, allows to reduce from the general case to that of hypographs. A second Ehrhard symmetrization, towards the plane, allows to reduce to \emph{non-increasing} planar hypographs, denoted $L_{\theta}$ below. Lastly, the (two-dimensional) Gaussian isoperimetric inequality, together with the assumption $\gamma_n(K)\geq \frac{1}{2}$, allows to reduce to only dealing with (decreasing) half-planes.

The advantage of the gaussian definition $K\circ u$ over its euclidean counterpart $K*u$ is that it easily gives the first two reductions (steps 1-2 below), and that the transform $T_r K=K \circ r e_n$ behaves well with respect to $\mathcal{E}_2$ (namely $\mathcal{E}_2 \circ T_r(\Lambda)=T_r\circ\mathcal{E}_2(\Lambda)$ for hypographs $\Lambda \subset \R^n$, see Claim \ref{claim2}).

\subsection{Proof that Lemma \ref{halfplanelem} implies Theorem \ref{gaussianmono}}

Let $K$ be a closed convex subset of $\R^n$ such that $\gamma_n(K) > 0$.  
We begin by showing the existence of a planar hypograph $L = L_{\theta}$, associated to a certain concave non-increasing function $\theta$, such that
\[
\gamma_n(K) = \gamma_2(L), \qquad 
\gamma_n(K \circ (r e_n)) \;\geq\; \gamma_2(L \circ (r e_2)).
\]
To achieve this (see steps 1 and 2 below), we perform two successive Ehrhard symmetrizations (setting $\Lambda=\mathcal{E}(K)$, and $L=\mathcal{E}_2(\Lambda)$), in order to reduce to $\R^2$.  
The definition of $K \circ u$ (see (\ref{dfn:gauss})) is such that it interacts well with the operators $\mathcal{E}$ and $\mathcal{E}_2$, making the reduction from $K$ to $L$ easy.

Next, we show the existence of a half-plane $H \subset \R^2$ such that
\[
\gamma_2(H) \;\geq\; \gamma_2(L), 
\qquad 
\gamma_2(H \circ (r e_2)) \;\leq\; \gamma_2(L \circ (r e_2)),
\]
which completes the reduction from $K$ to $H$, i.e., from Theorem~\ref{gaussianmono} to Lemma~\ref{halfplanelem}.  
For this last step, we use the Gaussian isoperimetric inequality together with the assumption $\gamma_n(K) \geq \frac{1}{2}$.

Let us recall two notations here, for the Ehrhard symetrizations $\Lambda=\mathcal{E}(K)$ and $L=\mathcal{E}_2(\Lambda)$, and their associated concave functions $\psi :\R^{n-1}\to \overline{\R}$ and $\theta : \R \to \overline{\R}$ :
\begin{align*}
\Lambda&=\Lambda_{\psi}=\{(y,z)\in \R^n | z\leq \psi(y)\} \hspace{2mm} &\text{where }\psi(y):=\Phi^{-1}(\gamma_1((K-y)\cap \R e_n))=:\Phi^{-1}(\gamma_1(I_y)) \hspace{2mm} ; \\
L&=L_{\theta}=\{(x,z)\in \R^2 | x\leq \theta(z)\} \hspace{1mm} &\text{where }\theta(z):=\Phi^{-1}(\gamma_{n-1}(\Lambda_z))=:\Phi^{-1}(\gamma_{n-1}((\Lambda-ze_n)\cap e_n^{\perp})) \hspace{1mm} ;
\end{align*}
in other words, since $\Lambda=\Lambda_{\psi}$, one has $\theta(z)=\Phi^{-1}(\gamma_{n-1}(\{y : \psi(y) \geq z\}))$.

Let us now give more details. Firstly, we may, without loss of generality, assume that
$u = r e_n,$ (with $r = \|u\|_2 > 0$).
Indeed, for any $\rho \in O(n)$, one has $\rho(K \circ u) = \rho(K) \circ \rho(u)$.  
Therefore, by rotational invariance of $\gamma_n$, and choosing $\rho$ such that $\rho(u) = u' = \|u\|_2 e_n$, we can assume $u = r e_n$.

\vspace{1mm}
\begin{equation}
\label{eqn:step1}
\text{\textbf{Step 1}. Let $\Lambda = \mathcal{E}(K)$. We show that } 
\gamma_n(K \circ (re_n)) \,\geq\, \gamma(\Lambda\circ (re_n)).
\tag{$\clubsuit$} 
\end{equation}

This is a direct consequence of Definition (\ref{dfn:gauss}) and of the one-dimensional Gaussian isoperimetric inequality. Indeed, one has $K\circ (re_n)=(K+[-r,r]e_n)\cap (Z_K+\R e_n)$ and, for the same convex set $Z_K\subset e_n^{\perp}$, one has $\Lambda \circ (re_n)=(\Lambda+[-r,r]e_n)\cap (Z_K+\R e_n)=(\Lambda + re_n) \cap (Z_K+\R e_n)$ with $Z_K=\{y\in e_n^{\perp} | \gamma_1(I_y)\geq \gamma_1([-r,r])\}=\{y\in e_n^{\perp} | \psi(y)\geq -w_r\}$, where $\Phi(-w_r)=\gamma_1([-r,r])$ defines $w_r$.

Denote $\pi_n$ the (orthogonal) reflection, in $\R^n$, with kernel $e_n^{\perp}$. If $K$ or $-K$ is a hypograph, then $\Lambda=K$ or $\Lambda=\pi_n(K)$, and so (\ref{eqn:step1}) holds with equality (since $\pi_n(\tilde{K} \circ re_n)=(\pi_n \tilde{K}) \circ re_n$, for any convex $\tilde{K}\subset \R^n$). Otherwise, the intervals $I_y=(K-y)\cap \R e_n$ are bounded (for all $y$), we denote $I_y=:[a_y,b_y] e_n$. By definition, one has $(\Lambda-y)\cap \R e_n=(-\infty, c_y] e_n$ with $\gamma_1((-\infty,c_y])=\gamma_1([a_y,b_y])$. The one-dimensional Gaussian isoperimetric inequality gives : $\gamma_1((-\infty, c_y+r])\leq \gamma_1(a_y-r,b_y+r])$. Integrating the latter over $y\in Z_K$ yields
\[\gamma_n(K\circ re_n)=\int_{Z_K} \gamma_1([a_y-r,b_y+r])d\gamma_{n-1}(y)\geq \int_{Z_K} \gamma_1((-\infty,c_y+r])d\gamma_{n-1}(y)=\gamma_n(\Lambda \circ re_n). \]
---
\begin{equation}
\label{eqn:step2}
\text{\textbf{Step 2}. Let $L=L_{\theta}= \mathcal{E}(\Lambda)$. We show that } 
\gamma_2(L \circ (re_2)) = \gamma_n(\Lambda\circ (re_n)).
\tag{$\spadesuit$} 
\end{equation}
In fact, Claim \ref{claim2} below implies that 
$\mathcal{E}_2(\Lambda \circ (r e_n)) = L \circ (r e_2),$
and since $\mathcal{E}_2$ preserves the Gaussian measure, (\ref{eqn:step2}) follows.
Let us introduce the notation 
\[T_r : \mathcal{K}_n \to \mathcal{K}_n, 
\qquad T_r(K) = K \circ (r e_n), \quad n \geq 2.
\]
(Strictly speaking, we should write $T_r^{(n)}$ for the map $T_r^{(n)}(K) = K \circ (r e_n)$ from $\mathcal{K}_n$ to itself. We slightly abuse notation and simply write $T_r$, letting the dimension of the ambient space implicit.)
\begin{claim}
\label{claim2}
  For any $n \geq 2$ and any concave $\psi : \R^{n-1} \to \overline{\R}$, one has
$\mathcal{E}_2 \circ T_r(\Lambda_\psi) \;=\; T_r \circ \mathcal{E}_2(\Lambda_\psi)$
where $\Lambda_{\psi}=\{(y,z)\in \R^n | z\leq \psi(y)\}$ denotes the hypograph of $\psi$.
\end{claim} 

\begin{proof}
We denote $\theta(z)=\Phi^{-1}(\gamma_{n-1}(\Lambda_z))=\Phi^{-1}(\gamma_{n-1}(\{y : \psi(y) \geq z\}
))$, so that $L_{\theta}=\mathcal{E}_2(\Lambda_{\psi})$. Let us also denote $\theta_r$ and $\psi_r$ the concave functions such that $\Lambda_{\psi} \circ re_n=\Lambda_{\psi_r}$, and such that $L_{\theta}\circ re_2=L_{\theta_r}$. In other words, $\theta_r$ and $\psi_r$ are given by
% ...
\[
\begin{minipage}{0.45\textwidth}
\[
\theta_r(z) :=
\begin{cases}
\theta(z-r)\phantom{,} & \text{if } z \geq -w_r + r,\\[2mm]
\theta(-w_r)\phantom{,} & \text{otherwise.}
\end{cases}
\]
\end{minipage}
\hfill
\begin{minipage}{0.45\textwidth}
\[
\text{and }\hspace{2mm}\psi_r(y) :=
\begin{cases}
\psi(y) + r\phantom{,} & \text{if } \psi(y) \geq -w_r,\\[1mm]
-\infty\phantom{,} & \text{otherwise.}
\end{cases}
\]
\end{minipage}
\]
Finally, let us denote $\theta'_r$ the concave function such that $\mathcal{E}_2(\Lambda_{\psi_r})=:L_{\theta'_r}$. By definition of $\mathcal{E}_2$, one has
\[\theta'_r(z)=\Phi^{-1}(\gamma_{n-1}(\{y | \hspace{1mm} \psi_r(y) \geq z\}))=\Phi^{-1}\left(\gamma_{n-1}\left(\left\{y | \hspace{1mm}\psi(y) \geq \max\{z-r, -w_r\}\right\}\right)\right)\]
and so $\theta'_r=\theta_r$, and hence $\mathcal{E}_2(\Lambda_{\psi} \circ re_n)=\mathcal{E}_2(\Lambda_{\psi_r})=L_{\theta'_r}=L_{\theta_r}=L_{\theta} \circ re_2$ as claimed.
\end{proof}
---

\textbf{Step 3}. We show that there exists an affine halfplane $H\subset \R^2$ such that 
\begin{equation}
\label{eqn:step3}
\gamma_2(H)\geq \gamma_2(L), \quad \text{ and} \hspace{2mm}
 \gamma_2(L \circ (r e_2)) \geq \gamma_2(H \circ (r e_2)).
\tag{$\blacklozenge$} 
\end{equation}

If $L=L_{\theta}$ is already a horizontal half-plane, i.e. if, for some $y_0\in \R$,
$L = \{(x,y) : y \leq y_0\},$
then  just take $H=L$. We can exclude this case, and therefore assume that $\theta(\R) \subset \R\cup \{-\infty\}$.  

Note that the case $\theta(-w_r) = -\infty$ is excluded, since this would imply ($\theta$ being non-increasing) that $L \subset H_0 := \{(x,y) : y \leq -w_r\}$, and hence
$\gamma_2(L) \;\leq\; \Phi(-w_r) = \gamma_1([-r,r]) \;\leq\; \tfrac{1}{3},$
contradicting the assumption $\gamma_2(L) = \gamma_n(K) \geq \tfrac{1}{2}$.

\begin{figure}
\centering
\includegraphics[width=0.90\linewidth]{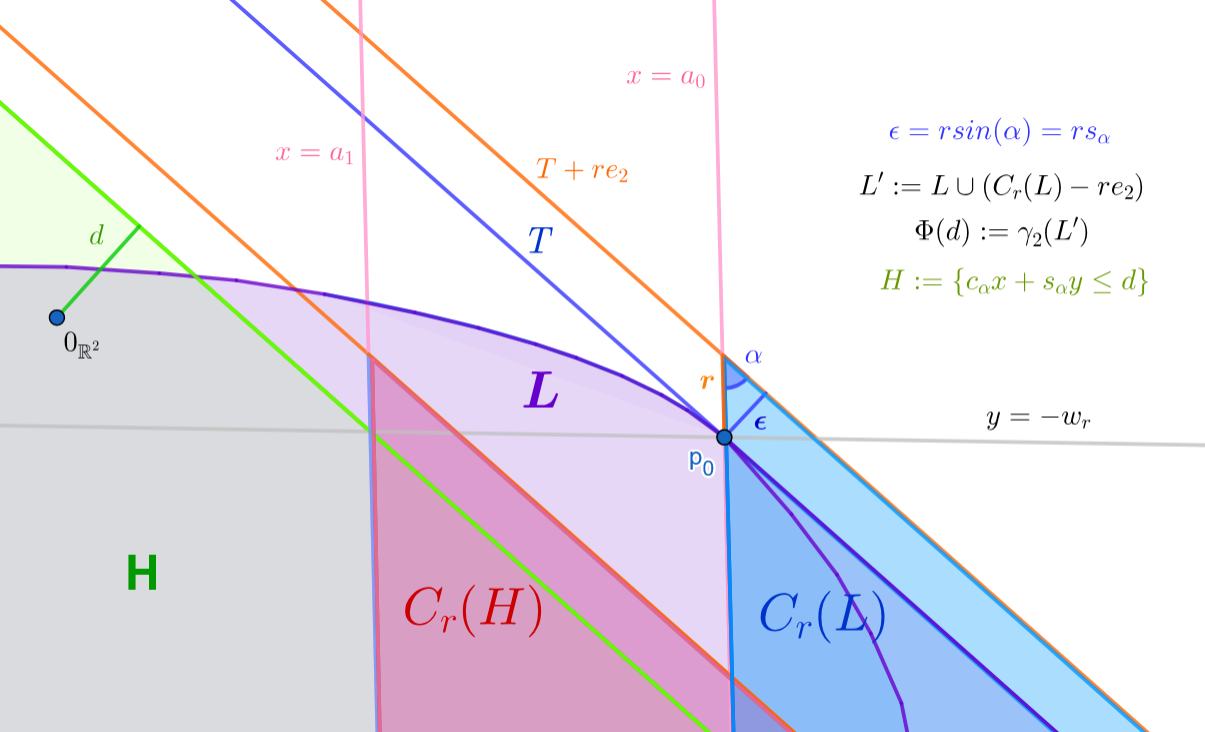}
\caption{\label{fig:step3} Visualizing the halfplane $H$, and the cones $C_r(L)$ and $C_r(H)$, for Step 3.}
\end{figure}

Thus set $a_0 := \theta(-w_r) \in \R$, and let $T$ be a line tangent to $L$ at $p_0 = (a_0, -w_r)$.  
(If there are several such tangent lines, choose the most vertical one, i.e. the one closest to $\partial L \cap \{y < -w_r\}$.)  

In case $T$ is vertical, meaning that $\theta$ is constant on $(-\infty, -w_r]$, we have $L \circ (r e_2) = L + r e_2 \supseteq L,$
and therefore 
we can skip\footnote{even though one can easily prove the existence of some $H$ as in (\ref{eqn:step3}) in that case as well} \textbf{step 3} and directly conclude (using only \textbf{steps} 1, 2) that \[\gamma_n(K\circ (re_n))-\gamma_n(K) \geq \gamma_2(L \circ (r e_2)) - \gamma_2(L) \geq 0.\] Therefore we may assume that
$\partial L \cap \{y = -w_r\} = \{p_0\},$
and that the line $T$ tangent to $L$ at $p_0$ has negative slope $- \tfrac{1}{\tan(\alpha)}$ for some $0 < \alpha < \tfrac{\pi}{2}$.
Denote $(c,s):=(\cos(\alpha),\sin(\alpha))\in [0,1]^2$. Thus $T$ has equation $cx+sy=D$ for some $D\geq \Phi^{-1}(\gamma_2(L)) \geq 0$.
Define the concave function $\tilde{\theta}\geq \theta$ by 
$$\tilde{\theta}(y)=\theta(y) \text{ if }y\geq -w_r  \quad \text{and } \hspace{1mm}\tilde{\theta}(y)=\frac{D-sy}{c}\quad  \text{ if } y\leq -w_r $$
so that $L\subseteq L':=L_{\tilde{\theta}}=\{(x,y) : x\leq \tilde{\theta}(y)\}$.
Observe that $L'\circ (re_2)=L\circ (re_2)=(L+re_2) \cap \{x\leq a_0\}$.

Define $H:=\{cx+sy \leq d\}$ with 
$d:=\Phi^{-1}(\gamma_2(L'))) \geq \Phi^{-1}(\gamma_2(L)) \geq 0$.
We shall argue that this $H$ satisfies $\gamma_2(H\circ (re_2)) \leq \gamma_2(L'\circ (re_2))$, and hence (\ref{eqn:step3}).

Denote $\varepsilon:=rs=r\sin(\alpha)$ so that $H+re_2=H+\varepsilon B_2^2=:H^{\varepsilon}$. Denote $C_r(L)=\{(x,y) : ca_0\leq cx\leq D+rs-sy\}$ the cone on the right of $\{x=a_0\}$ and below $T+re_2$ ; so that $L\circ (re_2)=(L'+re_2) \setminus C_r(L)$.

Denote $C_r(H)=\{(x,y) : ca_1\leq cx\leq d+rs-sy\}$ where $a_1\leq a_0$ is such that $ca_1-sw_r=d$ (i.e. $(a_0-a_1)\cos(\alpha)=D-d \geq 0$). In other words, $C_r(H)$ is the cone such that $H\circ (re_2)=(H+re_2)\setminus C_r(H)$. Observe that $C_r(L)=C_r(H)+(a_0-a_1) e_1$, and that $a_1=\frac{d+sw_r}{c}\geq 0$ (since $d\geq \Phi^{-1}(\gamma_2(L))=\Phi^{-1}(\gamma_n(K))\geq 0$, and since $w_r \geq r \geq 0$, as $r\leq \Phi^{-1}(2/3)$). Since $\frac{d\gamma_1(x)}{dx}$ is decreasing on $[0,+\infty)$, we deduce that $\gamma_2(C_r(L))\leq \gamma_2(C_r(H))$.

We claim that $L'+\varepsilon B_2^2 \subset L'+r e_2$ (see Claim \ref{exoL1} in Appendix, for a proof). It follows that
\begin{align}
\gamma_2(L\circ re_2)=\gamma_2(L'+re_2)-\gamma_2(C_r(L)) & \geq \gamma_2(L'+\varepsilon B_2^2) -\gamma_2(C_r(L)) \\
& \geq \gamma_2(H+\varepsilon B_2^2) -\gamma_2(C_r(H)) \\
&=\gamma_2(H+re_2) -\gamma_2(C_r(H))=\gamma_2(H \circ re_2).
\end{align}
where (5) follows from Claim \ref{exoL1}, (7) is by choice of $\varepsilon$ and by definition of $C_r(H)$ and (6) results from $\gamma_2(C_r(L))\leq \gamma_2(C_r(H))$ and from the two-dimensional Gaussian isoperimetric inequality.

\vspace{3mm}
Hence, it follows from (\ref{eqn:step1}, \ref{eqn:step2}, \ref{eqn:step3}) that:
\[
\gamma_n(K \circ (r e_n)) - \gamma_n(K) 
\geq
\gamma_n(\Lambda \circ (r e_n)) - \gamma_n(\Lambda)
=\gamma_2(L\circ (re_2))-\gamma_2(L)
\geq
\gamma_2(H \circ (r e_2)) - \gamma_2(H),
\]
and so $\gamma_n(K\circ u)-\gamma_n(K)\geq 0$ follows from Lemma~\ref{halfplanelem}.

\section{Monotonicity holds for half-planes}
\label{section:halfplane}

Therefore, to prove Theorem \ref{gaussianmono}, it only remains to prove the following Lemma.

\begin{lemma}
\label{planarlem}
Given $\alpha\in [0,\frac{\pi}{2}]$,and and let $d\geq 0$, denote  $(c,s):=(\cos(\alpha), \sin(\alpha))$, and $H=H(d,\alpha)=\{(x,y) : cx+sy \leq d\} \subset \R^2$. Then for any $r\in [0,\frac{1}{7}]$, $\gamma_2(H\circ (r e_2)) \geq \gamma_2(H)$.
\end{lemma}
 
Note that, by invariance of $\gamma_2$ under reflections, one has $(\pi K) \circ (\pi u)=\pi(K\circ u)$ for any $K\in \mathcal{K}_2$, $u\in \R^2$, and any planar relfection $\pi$. This, and the obvious fact (see Definition (\ref{dfn:gauss})) that $K\circ u=K \circ (-u)$, for any $K,u$, shows that
Lemma \ref{halfplanelem} and Lemma \ref{planarlem} are equivalent.
\vspace{1mm}


Before giving a proof of Lemma \ref{planarlem}, let us remark that if $0\leq r\leq \frac{1}{7}\leq \Phi^{-1}(2/3)$, then in particular $\gamma_1([-r,r])<\frac{1}{3}$, which ensures that $w_r-r>0$, where $-w_r=\Phi^{-1}(\gamma_1([-r,r])$.

\begin{proof}[Proof of Lemma \ref{planarlem}]
Let us first exclude two simple cases. If $\alpha=0$, i.e. if $H=\{(x,y) : x\leq d\}$, then $H\circ (re_2)=H$ for any $r\geq 0$, so there is nothing to show. If $\alpha=\frac{\pi}{2}$, then $H=\{(x,y) : y\leq d\}$, and since $r\leq \frac{1}{7} < \Phi^{-1}(1/2)$ (i.e. since $-w_r<0<d$), one has $H\circ (r e_2)=H+r e_2$, so $\gamma_2(H\circ (r e_2))-\gamma_2(H)=\Phi(d+r)-\Phi(d) \geq 0$. Therefore we shall assume that $\alpha \in (0,\frac{\pi}{2})$. We denote $-w_r := \Phi^{-1}(\gamma_1([-r,r])<-r \leq 0$.

We observe that $H\circ (r e_2)=(H+re_2) \setminus C_r$, with $C_r=\{(x,y) : cx_0 \leq cx \leq d+rs-sy\}$, where $x_0=x_0(d,\alpha,r)$ is given by $cx_0-sw_r=d$. Therefore
\begin{align*}
    \gamma_2(H\circ (re_2))-\gamma_2(H)=\gamma_2((H+r e_2)\setminus H) -\gamma_2(C_r).
\end{align*}
On the one hand, $\gamma_2((H+r e_2)\setminus H)=\Phi(d+rs)-\Phi(d) \geq \frac{se^{-(d+rs)^2/2}}{\sqrt{2\pi}} r$, since $d\geq 0$ (and $r\geq 0$). On the other hand, using the notation $T(A)=\Phi(-A)$, and the estimate $T(A)\leq \frac{e^{-A^2/2}}{A\sqrt{2\pi}}$ (for any $A>0$), one can upper bound $\gamma_2(C_r)$ as follows (with $y_x:=\frac{cx-d-rs}{s}$):
\begin{align*}
    \gamma_2(C_r)&=\int_{x_0}^{+\infty} \Phi\left(\frac{d+rs-cx}{s}\right) e^{-x^2/2} \frac{dx}{\sqrt{2\pi}} = \int_{x_0}^{+\infty} T(y_x) e^{-x^2/2} \frac{dx}{\sqrt{2\pi}} \\
    &\leq \int_{y_0}^{+\infty} \frac{e^{-y^2/2} e^{-x_y^2/2}}{y \sqrt{2\pi}} \frac{sdy}{c\sqrt{2\pi}} \hspace{3mm} \text{ with } x_y=\frac{d+rs+sy}{c} \text{ and } y_0=w_r-r>0.\\
\end{align*}
Now $y^2+\left(\frac{d+rs+sy}{c}\right)^2=\left(\frac{y+s(d+rs)}{c}\right)^2+(d+rs)^2$, and so changing for $z_y=\frac{y+s(d+rs)}{c}$ one finds
\begin{equation}
\label{eqn:uppboundcone}
      \gamma_2(C_r)\leq  \frac{se^{-(d+rs)^2/2}}{\sqrt{2\pi}} \int_{y_0}^{+\infty} \frac{e^{-z_y^2/2}}{w_r-r} \frac{dy}{c\sqrt{2\pi}}=  \frac{se^{-(d+rs)^2/2}}{\sqrt{2\pi}} \frac{T(z_0)}{w_r-r} \leq \frac{se^{-(d+rs)^2/2}}{\sqrt{2\pi}} \frac{T(w_r-r)}{w_r-r} 
\end{equation}
\begin{align*}
\text{ where, (since $d\geq 0$, and $c\leq 1$),} \hspace{2mm} z_0:=\frac{y_0+s(d+rs)}{c}=\frac{w_r}{c}-\frac{r}{c}+\frac{sd}{c}+\frac{s^2r}{c} \geq \frac{w_r}{c} -rc \geq w_r-r , 
\end{align*}
justifying the last inequality in (\ref{eqn:uppboundcone}). Hence to show that $\gamma_2(H\circ (re_2))-\gamma_2(H)=\Phi(d+rs)-\Phi(d)-\gamma_2(C_r) \geq 0$, it suffices to argue that $T(w_r-r)=\Phi(-w_r+r) \leq r(w_r-r)$, for $0\leq r \leq \frac{1}{7}$.

Denote $\lambda_0>0$ the unique solution to $2+e^{-\lambda^2/2}=\lambda \sqrt{2\pi}$. Note that $(r\mapsto w_r-r)$ is strictly decreasing on $(0,+\infty)$. Denote $r_0>0$ the unique solution to $w_r-r=\lambda_0$. Thus, for $0\leq r \leq r_0$ :
\begin{align*}
    \Phi(-w_r+r)&\leq \Phi(-w_r)+\frac{re^{-(w_r-r)^2/2}}{\sqrt{2\pi}}=\gamma_1([-r,r])+\frac{re^{-(w_r-r)^2/2}}{\sqrt{2\pi}} 
    \leq r \frac{2+e^{-(w_r-r)^2/2}}{\sqrt{2\pi}} \leq r (w_r-r).
\end{align*}
Numerical simulations show that $r_0>0.149>\frac{1}{7}$.
\end{proof}

%

\section{Appendix: two remarks and a claim}
\subsection{Remarks on conjectural variants of Theorem \ref{gaussianmono}}
A natural question in view of Theorem \ref{gaussianmono}, is the following:
\begin{question}
\label{maudquestion}
    Does there exist an increasing function $r: (0,1)\to (0,+\infty)$ such that for any $n\geq 2$, and for any convex body $K\in \mathcal{K}_n$, if $u\in r(\gamma_n(K)) B_2^n$, then $\gamma_n(K\circ u)\geq \gamma_n(K)$?
\end{question}
In other words, does there exist, for each $0<p<1$, some $r_p>0$ such that for any convex body $K$ with $\gamma_n(K)\geq p$, and any $u\in r_p B_2^n$, one has $\gamma_n(K\circ u)\geq \gamma_n(K)$? One may ask the same question but restricting to centrally symmetric convex bodies ($K=-K$): let us denote $r_p^{(s)}$ the greatest such $r>0$, if it exists, and $r_p\leq r_p^{(s)}$ the maximal $r(p)$ which would answer Question \ref{maudquestion}. Section \ref{section:reductions} shows that determining $r_p$ is a planar question. More precisely, Steps 1 and 2 show that, for any fixed $p\in (0,1)$ and $r_0>0$ such that $\gamma_1([-r_0,r_0])\leq p$, the following statements are equivalent:
\begin{itemize}
\item[(i)] for any $n\geq 2$, and any $(K,u)\in \mathcal{K}_n\times (r_0 B_2^n)$, if $\gamma_n(K)\geq p$ then $\gamma_n(K\circ u)\geq \gamma_n(K)$ ;
\item[(ii)] for any non-increasing concave function $\theta :\R \to \R \cup \{-\infty\}$, if $L_{\theta}=\{(x,y)  :x\leq \theta(y)\}$ has measure $\gamma_2(L_{\theta})\geq p$, then for any $r\in [0,r_0]$, $\gamma_2(L_{\theta} \circ (re_2))\geq \gamma_2(L_{\theta})$.
\end{itemize}
However, as soon as $p<\frac{1}{2}$, it isn't clear whether $(ii)$ is equivalent with the following, weaker statement:
\begin{itemize}
\item[(iii)] for any half-plane $H\subset \R^2$, for any $r\in [0,r_0]$, if $\gamma_2(H)\geq p$, then $\gamma_2(H\circ (re_2))\geq \gamma_2(H)$.
\end{itemize}
Computations similar to those of section \ref{section:halfplane} show that $r_p \leq \exp(-c \Phi^{-1}(p)^2 e^{\Phi^{-1}(p)^2})$ (almost vertical halfplanes $H$ yield this upper bound). Note that $r_p^{(s)}$ is at most linear in $p$ (near $0$), because of symmetric slabs. We conjecture that $r_p^{(s)}$ is indeed linear in $p$ (near $0$). For the matter of upper bounding $\text{vb}(B_2^n, K)$ for symmetric convex bodies $K$ of small measure, the upper bound $r_p^{(s)} \leq c_1 p$ tells us that the above method, which yields $\text{vb}(B_2^n, K) \leq (r^{(s)}(\gamma_n(K))^{-1}$, doesn't give any bound better than the one directly deduced from Theorem \ref{Ban5K} and from the $S$-inequality (\cite{LO99}). 

Indeed, denote $\Psi(x)=2\Phi(x)-1=\gamma_1([-x,x])$, and $\Psi^{-1} : [0,1) \to [0,+\infty)$ its inverse. In particular, $\Psi^{-1}(p) \geq c_0 p$ for all $p$ (with $c_0=\sqrt{\frac{\pi}{2}}$), and $\Psi^{-1}(p)\leq c_1 p$ for $p\in (0,\frac{1}{2}]$ (take $c_1=c_0 e^{x_1^2/2}$, where $x_1=\Psi^{-1}(1/2)$). If $K=-K$, then the $S$-inequality states that $\gamma_n(tK)\geq \Psi(t\Psi^{-1}(p))$ for any $t\geq 1$. In particular, if ($K=-K$ and) $\gamma_n(K)=p\in (0,\frac{1}{2})$, then $\gamma_n(tK)\geq \frac{1}{2}$ for $t\geq \frac{\Psi^{-1}(1/2)}{\Psi^{-1}(p)}$. Taking $t=c_2 p^{-1}$ (with $c_2=c_0^{-1} x_1$), Theorem \ref{Ban5K} gives that $\text{vb}(B_2^n, K)=t\text{vb}(B_2^n, tK) \leq 5t=c_3 p^{-1}$ ; while the bound $\text{vb}(B_2^n, K)\leq \frac{1}{r_p^{(s)}}$ (yielded by Banaszcyk's argument recalled on page $2$) cannot yield better than $\text{vb}(B_2^n, K)\leq c_1^{-1} p^{-1}$, since $r_p^{(s)}\leq \Psi^{-1}(p)\leq c_1 p$ (for $p\leq \frac{1}{2}$).

At the other end, note that, by homogeneity, $\text{vb}(B_2^n, (1+\epsilon)\sqrt{n} B_2^n)=\frac{1}{(1+\epsilon)\sqrt{n}}\text{vb}(B_2^n,B_2^n)\geq 1-\epsilon$, while, for any given $\delta>0$, $\gamma_n((1+\epsilon)\sqrt{n}B_2^n)\geq 1-\delta$ for large enough $n$, by concentration. It results that $\lim_{p\to 1} r_p^{(s)}\leq 1$. In  words, if $u\in \R^n$ is such that monotonicity of the gaussian measure under Banaszczyk transform (either $K*u$ or $K\circ u$) holds for a class of symmetric convex bodies which contains Euclidean balls, then $||u||_2\leq 1$.

\subsection{Appendix : closing a small gap from Step (\ref{eqn:step3})}
We left an elementary claim unproven in section \ref{section:reductions}, within the last reduction (from planar hypographs $L_{\theta}$ to half-planes). We restate it here, and provide a proof. Denote $(e_1,e_2)$ the canonical basis of $\R^2$.

\begin{claim}
\label{exoL1}
Let $\theta : \R \to \R \cup\{-\infty\}$ be non-increasing and concave on $\R$, and affine on $(-\infty, -w]$ for some $w\in \R$. Denote $L_{\theta}=L:=\{(x,y) \in \R^2 : x\leq \theta(y)\}$. Assume $\theta$ has slope $-\tan(\alpha)$ on $(-\infty, -w]$, for some $0<\alpha<\frac{\pi}{2}$ (see Figure \ref{fig:inclusion}). Then for any $\varepsilon>0$, $L^{\varepsilon}:=L+\varepsilon B_2^2\subseteq L+\frac{\varepsilon}{\sin(\alpha)} e_2$.
\end{claim}


If $K
\subset R^n$ is closed and convex, then its support function $h_K : \R^n \to \R$ is the function $h_K(x)=\sup \{\langle y,x\rangle | y\in K\}$. Since $h_K$ is positively homogeneous, one often simply defines $h_K$ as a function on the sphere: $h_K(u)=\sup \{\langle y,u\rangle | y\in K\}$, $u\in \mathbb{S}^{n-1}$. For instance if $H$ is an affine halfspace, i.e. $H=H(u_0,d):=\{x\in \R^n : \langle x, u_0\rangle \leq d\}$ for some $u_0\in \mathbb{S}^{n-1}$, $d\in \R$, then $h_H(u_0)=d$ and $h_H(u)=+\infty$ for any $u\in \mathbb{S}^{n-1} \setminus \{u_0\}$. For any closed convex set $K\subset \R^n$, one has $K=\cap_{u\in \mathbb{S}^{n-1}} H(u, h_K(u))$, the non-trivial inclusion is given for instance by Hahn-Banach theorem. In particular, given two closed convex sets $K, L \subset \R^n$, the inclusion $K\subset L$ holds if and only if $h_K(u)\leq h_L(u)$ for all $u\in \sphere$.

Support functions behave well with Minkowski sums of convex sets : if $A, B\in \mathcal{K}_n$, then their sum $A+B=\{a+b \quad| a\in A,b\in B\}$ has support function $h_{A+B}=h_A + h_B$. If $B=B_2^n$, then $h_B=1$ on $\sphere$, while if $B=\{x\}$ for some $x\in \R^n$, then $h_B(u)=\langle x,u \rangle$ for all $u$.

\begin{proof}[Proof of Claim \ref{exoL1}]
 Let $(e_1, e_2)$ denote the canonical basis of $\R^2$. Denote $L_{\theta}=L:=\{(x,y) \in \R^2 : x\leq \theta(y)\}$, with $\theta$ as in Claim \ref{exoL1}. We aim at showing that $L+rs B_2^2:=L+r\sin(\alpha) B_2^2 \subset L+r e_2$ for any $r>0$. It suffices to argue that $h_L(u_{\beta})+r \sin(\alpha) \leq h_L(u_{\beta})+r\sin(\beta)$ for any $\beta \in [0.2\pi]$, where $u_{\beta}:=\cos(\beta) e_1+\sin(\beta) e_2$. Since $\theta$ is concave and non-increasing, it is clear that $h_L(u_{\beta})=+\infty$ unless $\beta\in [0,\frac{\pi}{2}]$. Under the additional assumption on $\theta$, note that, within the halfplane $\{y\leq -w\}$, the boundary of $L$ coincides with a line $T$ perpendicular to $u_{\alpha}$, so it is clear that $h_{L}(u_{\beta})=+\infty$ when $0\leq \beta \leq \alpha$. The desired inequality clearly holds in the remaining range, when $\beta\in [\alpha, \frac{\pi}{2}]$.
\end{proof}

\begin{figure}
\centering
\includegraphics[width=0.90\linewidth]{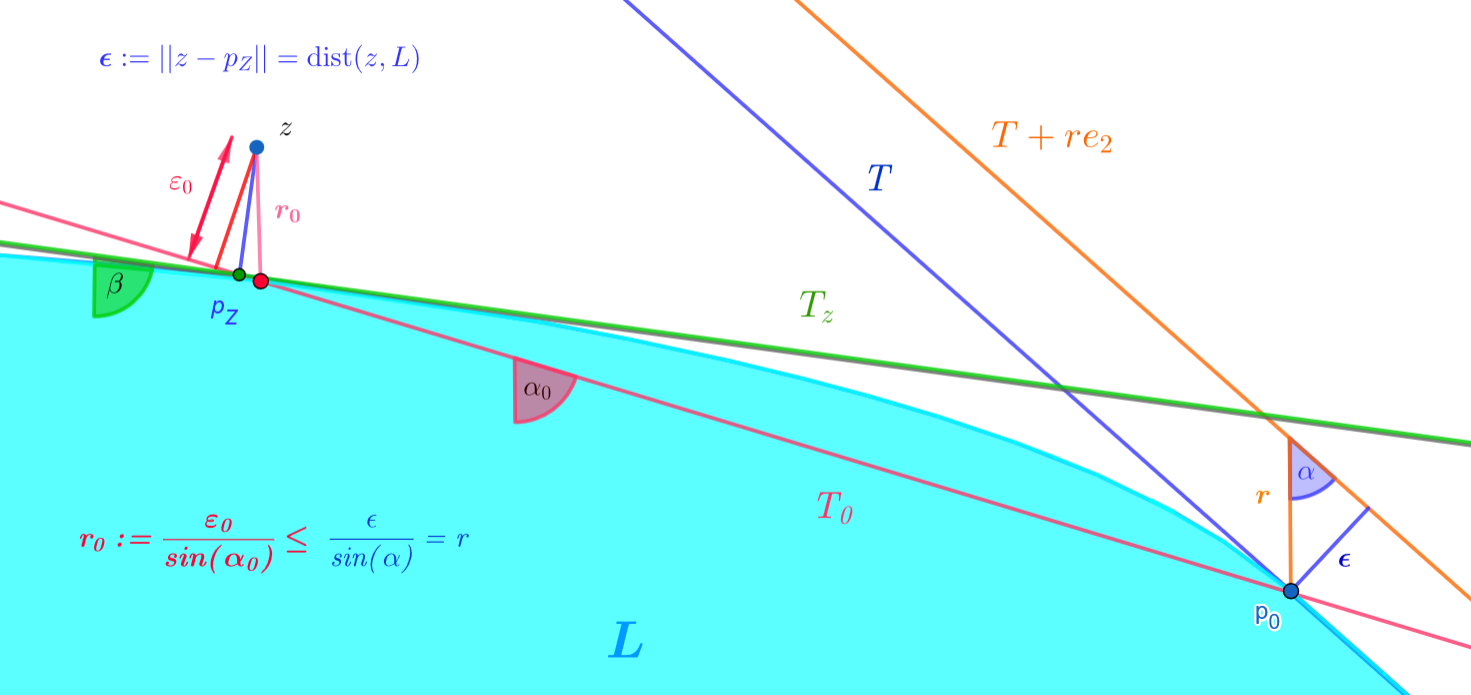}
\caption{\label{fig:inclusion} A proof by picture of Claim \ref{exoL1}}
\end{figure}

\textbf{Acknowledgments.} The author would like to thank Matthieu Fradelizi, Eli Putterman, and Kasia Wyczesany for their comments on a preliminary draft, which helped improve the presentation of this paper.  The author is indebted to Piotr Nayar and to Stanislaw Szarek for helpful discussions.

\bibliographystyle{plain}
\bibliography{Ban7K.bib}

\end{document}